\documentclass[12pt]{amsart}

\usepackage{newtxtext,newtxmath} 
\usepackage{amsfonts}
\usepackage{mathtools}

\newtheorem{theorem}{Theorem}[section]
\newtheorem{corollary}{Corollary}[theorem]
\newtheorem{proposition}{Proposition}[section]

\newtheorem{lemma}{Lemma}[section]

\newcommand{\E}[1]{\mathbb{E}\left[ #1 \right] }
\newcommand{\Var}[1]{\mathbb{V}\mathrm{ar}\left( #1 \right) }

\newcommand{\Prob}[1]{\mathbb{P}\left\lbrace #1 \right\rbrace}

\newcommand{\dx}[1]{ d#1 }

\newcommand{\Event}[1]{\left\lbrace #1 \right\rbrace }

\newcommand{\Pmatrix}[2]{\boldsymbol{P}_{#1,#2}}
\newcommand{\Qmatrix}[2]{\boldsymbol{Q}_{#1,#2}}
\newcommand{\Ones}[2]{\underset{(#1 \times #2)}{\boldsymbol{1}}}

\newcommand{\pink}[2]{\boldsymbol{\pi}_{#1}^{(#2)}}
\newcommand{\given}{\,|\,}
\newcommand{\Pk}[1]{P^{\langle #1 \rangle}}

\newcommand{\Indicator}[1]{\llbracket #1 \rrbracket}
\newcommand{\nthCoeff}{[z^n] \,}

\begin{document}  

\title{A class of random recursive tree algorithms with deletion}
\author{Arnold T. Saunders, Jr.}
\address{Department of Statistics, The George Washington University, Washington, DC 20052}
\email{arnold\_saunders@gwu.edu}
\date{\today}

\begin{abstract}
We examine a discrete random recursive tree growth process that, at each time step, either adds or deletes a node from the tree with probability $p$ and $1-p$, respectively. Node addition follows the usual uniform attachment model. For node removal, we identify a class of deletion rules guaranteeing the current tree $T_n$ conditioned on its size is uniformly distributed over its range. By using generating function theory and singularity analysis, we obtain asymptotic estimates for the expectation and variance of the tree size of $T_n$ as well as its expected leaf count and root degree. In all cases, the behavior of such trees falls into three regimes determined by the insertion probability: $p < 1/2$, $p = 1/2$ and $p > 1/2$. Interestingly, the results are independent of the specific class member deletion rule used.


\end{abstract}

\maketitle

\section{Introduction}

Tree evolution algorithms supporting both node insertion and deletion are notoriously hard to analyze. Jonassen and Knuth showed deriving the distribution of a mere three-node random binary search tree after a finite series of repeated insertions and deletions required Bessel functions and solving bivariate integral equations. In their words, ``the analysis ranks among the more difficult of all exact analyses of algorithms...the problem itself is intrinsically difficult \cite{JONA78}.'' Panny later chronicled a near half century of hopeful assumptions and poor intuition about the effect of deletions on binary search tree distribution \cite{PANN10}. In this paper, we study the effect of a class of deletion rules on the evolution of random recursive trees.

Random recursive trees are stochastic growth processes with diverse applications in modeling searching and sorting algorithms, the spread of rumors, Ponzi schemes and manuscript provenance \cite{SMYT95}.  The idea behind the model is straightforward. Starting from a root node labeled 1, we construct a tree one vertex at a time using sequentially labeled nodes. Each newly introduced node is ``randomly'' attached to an existing one in the tree. 

The insertion or attachment rule we use to construct the tree determines the distribution of $T_n$ over its range. For example, consider an insertion rule where each new node is attached to any of the existing ones with equal probability. The resulting trees are known as \textit{uniform recursive trees} or \textit{uniform attachment trees}. In this case, $T_n$ is uniformly distributed over the $n!$ possible recursive trees with $n+1$ nodes. Much research has gone into characterizing the limiting random variables and distributions of functionals on uniform recursive trees such as node degree \cite{DOND82,JANS05}, height \cite{PITT94}, leaf count \cite{NAJO82}, etc.

Motivated by work with random graph models incorporating both insertion and deletion rules (eg, \cite{BENN07,DEON07,GHOS13,JOHA18,ZHAN16}), we examine the less-studied application of such rules to tree evolution models. Specifically, we start with tree $T_0$ containing a single node labeled 1. At each time step $n \geq 1$, we either add an incrementally-labeled node to the tree with probability $p$ or delete an existing node with probability $q=1-p$. After a deletion, we reattach and relabel the remaining nodes so that $T_{n+1}$ is again a recursive tree. There is one exception to the preceding: we do not allow the tree to vanish. So if $T_n$ is the single node tree, it remains unchanged with probability $q$.

We always add nodes using the uniform attachment rule. We will however identify the class of deletion rules guaranteeing $T_n$, when conditioned on tree size, remains uniformly distributed over its range. We then, using singularity analysis of generating functions, provide a means for deriving the exact and asymptotic expressions of common functionals on $T_n$ such as tree size, leaf count and root degree. 

\section{Conditional Equiprobability}

A simplifying property of uniform attachment trees is the equiprobability of the range of $T_n$. Once we introduce deletion, this need not be the case. But if our choice of deletion rule could guarantee---conditioned on tree size---a uniformly distributed $T_n$, its analysis is greatly simplified. To specify the class of such deletion rules, we must first make concrete the notion of insertion and deletion rules.

Define the \textit{size} of a tree to be the number of nodes it possesses. Next define the \textit{stratum number} of a tree to be one less than its size. Let \textit{stratum} $k$ denote the set of all trees sharing the common stratum number $k$. Then stratum $k$ contains $k!$ trees and we can assign each one a unique integer identifier from 1 to $k!$ and arrange them in canonical order. We can now capture all the probabilities of transitioning from one of the $k!$ trees in stratum~$k$ to one of the $(k+1)!$ trees in stratum~$k+1$ (an insertion) in a single $k! \times (k+1)!$ conditional probability matrix $\Pmatrix{k}{k+1}$. Analogously, we can record the probabilities of transitioning from a stratum~$k+1$ tree to a stratum~$k$ tree (a deletion) by $\Qmatrix{k+1}{k}$. Insertion and deletion rules then are simply specifications of the form of matrices $\Pmatrix{k}{k+1}$ and $\Qmatrix{k+1}{k}$ for each $k \geq 0$, which we will call insertion and deletion matrices, respectively.

Since we are using uniform attachment as our insertion rule, each insertion matrix $\Pmatrix{k}{k+1}$ has the form
\begin{equation*}
\Pmatrix{k}{k+1} = \frac{p}{k+1} \, \boldsymbol{D},
\end{equation*} 
where $\boldsymbol{D}$ is a 0-1 matrix with row sums $k+1$ and columns sums 1. The exact placement of the 0s and 1s depends on the tree canonicalization used.

In the next theorem, we identify a necessary and sufficient condition on deletion matrices $\Qmatrix{k+1}{k}$ for conditional equiprobability and then establish the class of growth algorithms with that property.

\begin{theorem}
Conditioned on stratum number, each tree is equiprobable at time $n \geq 1$ if and only if
\begin{equation}
\Ones{1}{(k+1)!} \Qmatrix{k+1}{k} \propto \Ones{1}{k!} \hspace{5mm} (0 \leq k \leq n-1).
\label{eq:eqcolsums}
\end{equation}
Note the above is equivalent to requiring all column sums of $\Qmatrix{k+1}{k}$ to be identical. 

\begin{proof}
$(\Rightarrow)$ Assume that, conditioned on stratum number, each tree within a stratum is equiprobable at time $n \geq 1$. When $k=0$, the assertion is trivially true so let us assume $k \geq 1$. Let $T_n$ denote the recursive tree at time $n$ and $S_n$ its stratum number. Additionally, let $t$ be a stratum $k$ tree. Then by hypothesis, we have
\begin{equation*}
\Prob{T_n = t  \given  S_n=k} = \frac{1}{k!},
\end{equation*}
or equivalently
\begin{equation*}
\Prob{T_n = t} = \frac{1}{k!} \, \Prob{S_n=k}.
\end{equation*}
If we denote the distribution of $T_n$ within stratum $k$ by $\pink{n}{k}$, we can summarize this result succinctly with
\begin{equation}
\pink{n}{k} = \frac{1}{k!} \, \Prob{S_n=k} \Ones{1}{k!}.
\label{eq:eqprophypothesis}
\end{equation}
Consequently, by conditioning on the action (ie, insertion or deletion) at time $n$, we can express the distribution of the $k$th stratum ($1 \leq k \leq n-1$) at time $n+1$ by the following equality
\begin{multline}
\frac{1}{k!} \, \Prob{S_{n+1}=k} \Ones{1}{k!} = \pink{n+1}{k} = \pink{n}{k-1} \Pmatrix{k-1}{k} + \pink{n}{k+1} \Qmatrix{k+1}{k} =\\
\frac{1}{(k-1)!} \, \Prob{S_n=k-1} \Ones{1}{(k-1)!} \Pmatrix{k-1}{k} + \frac{1}{(k+1)!} \, \Prob{S_n=k+1} \Ones{1}{(k+1)!} \Qmatrix{k+1}{k}.
\label{eq:eqprobstratk}
\end{multline}
Next, by observing
\begin{equation*}
\Ones{1}{(k-1)!} \Pmatrix{k-1}{k} = \frac{p}{k} \Ones{1}{k!}, 
\end{equation*}
and noting the inequality
\begin{equation*}
\Prob{S_n=k} \geq \Prob{\text{$n-k$ deletions followed by $k$ insertions}} = q^{n-k}p^k > 0
\end{equation*}
holds whenever $0 \leq k \leq n$, the probabilities $\Prob{S_n=k-1}$ and $\Prob{S_n=k+1}$ in \eqref{eq:eqprobstratk}, subject to the given constraint $1 \leq k \leq n-1$, are positive, we can rearrange the terms on the left and right-hand sides of \eqref{eq:eqprobstratk} to obtain
\begin{equation*}
\Ones{1}{(k+1)!} \Qmatrix{k+1}{k} = \frac{k+1}{\Prob{S_n=k+1}} \left[ \Prob{S_{n+1}=k} - p \, \Prob{S_n=k-1} \right] \Ones{1}{k!}.
\end{equation*}
Finally, when $1 \leq k \leq n-1$, we have
\begin{equation*}
\Prob{S_{n+1}=k} = p \, \Prob{S_n=k-1} + q \, \Prob{S_n=k+1} > p \, \Prob{S_n=k-1}.
\end{equation*}
Thus $\Ones{1}{(k+1)!} \Qmatrix{k+1}{k} \propto \Ones{1}{k!}$ as claimed. \\

$(\Leftarrow)$ Assume $\Ones{1}{(k+1)!} \Qmatrix{k+1}{k} \propto \Ones{1}{k!}$ for arbitrary $0 \leq k \leq n-1$. By using mathematical induction on $n$, we show \eqref{eq:eqprophypothesis} holds. 

Since strata 0 and 1 contain only one tree each, the result is trivially true for $n=1$. Next assume it also holds for some arbitrary $n \geq 1$ and consider the case $n+1$. Since \eqref{eq:eqprophypothesis} always holds for $k =0$ and $k=n$ at time $n$, we can restrict our attention to $1 \leq k \leq n$ at time $n+1$. Now, by conditioning on the action at time $n$, we have
\begin{equation*}
\pink{n+1}{k} = \begin{dcases}
\pink{n}{k-1} \Pmatrix{k-1}{k} + \pink{n}{k+1} \Qmatrix{k+1}{k}, & 1 \leq k \leq n-1 \\
\pink{n}{k-1} \Pmatrix{k-1}{k}, & k=n.
\end{dcases}
\end{equation*}
Thus when $k=n$ we have
\begin{equation*}
\pink{n+1}{k} = \pink{n+1}{n} = \pink{n}{n-1} \Pmatrix{n-1}{n} \propto \Ones{1}{(n-1)!} \Pmatrix{n-1}{n} = \frac{p}{n!} \Ones{1}{n!} \propto \Ones{1}{k!}.
\end{equation*}
On the other hand, when $k < n$, we have
\begin{align*}
\pink{n+1}{k} &= \pink{n}{k-1} \Pmatrix{k-1}{k} + \pink{n}{k+1} \Qmatrix{k+1}{k} \\
& \propto \Ones{1}{(k-1)!} \Pmatrix{k-1}{k} + \Ones{1}{(k+1)!} \Qmatrix{k+1}{k} \propto \frac{p}{k!} \Ones{1}{k!} + \Ones{1}{k!} \propto \Ones{1}{k!}.
\end{align*}
In both cases we have $\pink{n+1}{k} = \beta \Ones{1}{k!}$ for some $\beta > 0$. Recalling
\begin{multline*}
\Prob{S_{n+1}=k} = \\
\sum_{t \in \mathcal{T}_k} \Prob{S_{n+1}=k,T_{n+1}=t} = \sum_{t \in \mathcal{T}_k} \Prob{T_{n+1}=t} = \pink{n+1}{k} \Ones{k!}{1} = \beta \Ones{1}{k!} \Ones{k!}{1} = \beta k!,
\end{multline*}
where $\mathcal{T}_k$ is the set of stratum $k$ trees, we conclude $\beta = \frac{1}{k!}\Prob{S_{n+1}=k}$ as desired.
\end{proof}

\label{thm:equiprob}
\end{theorem}

A consequence of Theorem~\ref{thm:equiprob} is that if our choice of deletion algorithm obeys \eqref{eq:eqcolsums} for arbitrary $n \geq 1$, then for all $n \geq 1$, all trees in the same stratum are equiprobable. We summarize this result with the following corollary.

\begin{corollary}[The Class of Conditional Equiprobable Growth Algorithms]
If the deletion matrix $\Qmatrix{k+1}{k}$ for a given growth algorithm satisfies \eqref{eq:eqcolsums} for arbitrary $k \geq 0$, then it supports conditional equiprobability. Moreover since any conditional equiprobable algorithm possesses this property, this criterion describes the class of such algorithms. Finally, this class is nonempty.

\begin{proof}
To show the class is not empty consider the ``last in, first out'' (\textit{LIFO}) deletion rule. When invoked, we delete the last node inserted into the tree. Then each row in an arbitrary deletion matrix $\Qmatrix{k+1}{k}$ contains exactly one nonzero entry, $q$. Each column of this matrix represents a stratum $k$ tree. By adding a node to this tree, we obtain $k+1$ trees in strata $k+1$. Hence each of the column sums is $(k+1)q$, satisfying condition \eqref{eq:eqcolsums}.
\end{proof}
\label{cor:eqprobclass}
\end{corollary} 

\section{Tree Size Generating Functions and Asymptotics}

Having established the class of deletion rules ensuring $T_n$ given $\Event{S_n=k}$ is equally likely to be any one of the $k!$ trees in stratum $k$, we next explore the distribution and moments of $S_n$, as well as those of several functions of $S_n$. 

\begin{proposition}
Let $P_{n,0}$ denote the probability the $n$th iteration of the algorithm generates the root tree. The ordinary generating function for the sequence $\{P_{n,0}; n\geq 0\}$ is
\begin{equation}
\Pk{0}(z) = \frac{2}{1-2qz+\sqrt{1-4pqz^2}}.
\label{eq:Pn0_GF}
\end{equation}

The asymptotic estimate for $P_{n,0}$ as a function of $p$ is
\begin{equation}
P_{n,0} \approx
\begin{dcases}
\frac{q-p}{q}, & p < \frac{1}{2} \\
\sqrt{\frac{2}{\pi n}}, & p = \frac{1}{2} \\
O\left( (2\sqrt{pq})^n n^{-3/2}\right) , & p > \frac{1}{2}.
\end{dcases}
\label{eq:Pno_asymp}
\end{equation}
Observe as $p$ approaches $1/2$ from the right, the quantity $2\sqrt{pq}$ goes to $1$, showing $P_{n,0}$ vanishes more slowly for probabilities $p$ close to $1/2$.

\begin{proof}

Noting $\Pk{0}(z)$ is the generating function of a biased excursion, that is, a biased random walk on the nonnegative integers starting and ending at zero, a simple modification of the generating function $M(z) = 2/(1-2z+\sqrt{1-4z^2})$ (see \textit{OEIS} \textbf{A001405}) for unbiased excursions gives us \eqref{eq:Pn0_GF}.

For an asymptotic estimate of $P_{n,0}$, we begin by noting the singularities of $\Pk{0}(z)$ occur at branch points $z = \pm 1/(2\sqrt{pq})$ and, if $p < 1/2$, also at a simple pole $z=1$. The branch points are on the unit circle when $p=1/2$. Otherwise, by a simple calculus argument, they are outside of it.

Consider the case $p < 1/2$. Since the branch points fall outside the unit circle, $\Pk{0}(z)$ is meromorphic within a disk of radius $R$, where $1 < R <  1/(2\sqrt{pq})$. Hence we can expand $\Pk{0}(z)$ about the simple pole $z=1$ to obtain the Laurent series representation
\begin{equation*}
\Pk{0}(z) = \frac{q-p}{q} \left( \frac{1}{1-z} \right)  + g(z),
\end{equation*}  
where $g$ is some function analytic at $z=1$ and therefore has radius of convergence $1/(2\sqrt{pq})$. Thus
\begin{equation*}
P_{n,0} = \frac{q-p}{q} + O\left((2\sqrt{pq} + \varepsilon)^n\right),
\end{equation*} 
where $2\sqrt{pq} < 1$ and $\varepsilon > 0$ is an arbitrarily small positive number \cite[Theorem IV.10, p 258]{FLAJ09}.

Next consider the case $p > 1/2$. Here the radius of convergence is determined by branch points on opposite sides of the imaginary axis. The function $\Pk{0}(z)$ is \textit{star-continuable} \cite[Theorem VI.5, p 398]{FLAJ09} and, in the vicinity of its singularities, we have
\begin{equation*}
\Pk{0}(z) =
\begin{dcases}
O\left( \sqrt{1-2\sqrt{pq}z} \right) , & z \rightarrow \frac{1}{2\sqrt{pq}} \\
O\left( \sqrt{1+2\sqrt{pq}z} \right), & z \rightarrow -\frac{1}{2\sqrt{pq}},
\end{dcases}
\end{equation*}
from which we obtain, by Big-Oh transfer  \cite[Theorem VI.3, p 390]{FLAJ09}, the asymptotic bound $P_{n,0} = O \left( (2\sqrt{pq})^n n^{-3/2} \right)$.

Finally for the case $p = 1/2$, if $M(z)$ is the generating function for the number of excursions of length $n$, then we have $\Pk{0}(z) = M(z/2)$ and therefore by Stirling's approximation
\begin{equation*}
P_{n,0} = \frac{1}{2^n} \binom{n}{\lfloor \frac{n}{2} \rfloor} \sim \sqrt{\frac{2}{\pi n}}.
\end{equation*}

\end{proof}

\label{prop:Pn0_GF}
\end{proposition}

We apply the same methodology to the remaining generating functions in this section. Unless the determination of the asymptotic estimates introduces something new, we will state the results without proof.

\begin{lemma}
Let $P_{n,1}$ denote the probability the $n$th iteration of the algorithm generates a stratum 1 tree. The ordinary generating function for the sequence $\{P_{n,1}; n\geq 0\}$ is
\begin{equation}
\Pk{1}(z) = \left( \frac{1}{qz} - 1 \right) \Pk{0}(z) - \frac{1}{qz}.
\label{eq:Pn1_GF}
\end{equation}

\begin{proof}
Let us condition on the last iteration to determine the probability of obtaining the root tree at iteration $n+1$. Doing so yields the recurrence relation
\begin{equation*}
P_{n+1,0} = q P_{n,0} + q P_{n,1}.
\end{equation*}
Rearranging terms so that $P_{n,1}$ is expressed in terms of $P_{n,0}$ and  $P_{n+1,0}$, then multiplying both size by $z^n$ and summing over $n \geq 0$ gives us
\begin{equation*}
\Pk{1}(z) = \frac{1}{q} \sum_{n \geq 0} P_{n+1,0} \, z^n - \Pk{0}(z) = \frac{1}{qz} \left[ \Pk{0}(z) - 1 \right]  - \Pk{0}(z) = \left( \frac{1}{qz} - 1 \right) \Pk{0}(z) - \frac{1}{qz}.
\end{equation*}

\end{proof}

\end{lemma}

\begin{proposition}
Let $P_{n,k}$ denote the probability of obtaining a stratum $k$ tree on the $n$th iteration. If we mark the stratum number with $u$, then the bivariate generating function $P(z,u)$ for the double sequence $\{ P_{n,k}; n \geq 0, k \geq 0 \}$ is
\begin{equation}
P(z,u) = \frac{q(1-u)z\Pk{0}(z)-u}{qz-u(1-puz)}.
\label{eq:Pnk_GF}
\end{equation}

\begin{proof}
For fixed $k$, $k \geq 0$, let us denote the generating function of the sequence $\{ P_{n,k}; n \geq 0  \}$ by $\Pk{k}(z)$ so that $P(z,u) = \sum_k \Pk{k}(z) \, u^k$. Then for $k \geq 1$, if we condition on the last iteration, we have the recurrence relation
\begin{equation*}
P_{n+1,k} = p \, P_{n,k-1} + q \, P_{n,k+1}. 
\end{equation*}
Multiplying both sides by $z^n$ and summing over $n \geq 0$ give us
\begin{align*}
\sum_{n \geq 0} P_{n+1,k} \, z^n &= p \Pk{k-1}(z) + q \Pk{k+1}(z) \\
\frac{1}{z} \left[ \Pk{k}(z) - P_{0,k} \right] &= p \Pk{k-1}(z) + q \Pk{k+1}(z) \\
\frac{1}{z}\Pk{k}(z) &= p \Pk{k-1}(z) + q \Pk{k+1}(z) & \text{since\ } P_{0,k}=0 \text{\ for all\ } k \geq 1.
\end{align*}
Now multiplying both sides by $u^k$ and summing over $k \geq 1$ yields
\begin{align*}
\frac{1}{z} \sum_{k \geq 1} \Pk{k}(z) u^k &= p \sum_{k \geq 1} \Pk{k-1}(z) u^k + q \sum_{k \geq 1} \Pk{k+1}(z) u^k \\
\frac{1}{z} \left[ P(z,u) - \Pk{0}(z) \right] &= pu P(z,u) + \frac{q}{u} \left[ P(z,u) - \Pk{0}(z) - u\Pk{1}(z) \right] \\
P(z,u) &= \frac{\left[ \frac{1}{z}-\frac{q}{u} \right] \Pk{0}(z) - q \Pk{1}(z)}{\frac{1}{z}-pu-\frac{q}{u}}.
\end{align*}
Finally, substituting the right side of \eqref{eq:Pn1_GF} for $\Pk{1}(z)$ and some simplification gives us \eqref{eq:Pnk_GF}.
\end{proof}

\end{proposition}

If we let $S_n$ denote the stratum number of a tree at time $n$, then $P_n(u) \equiv \nthCoeff P(z,u)$ is the probability generating function of $S_n$. Thus we immediately have $\E{S_n} = P_n'(1)$ and $\Var{S_n} = P_n''(1)+P_n'(1)-[P_n'(1)]^2$. This idea leads to the following generating functions for the first and second factorial moments of $S_n$ and asymptotic estimates for $\E{S_n}$ and $\Var{S_n}$.

\begin{proposition}
Let $S_n$ denote the stratum number of the tree generated by the $n$th iteration of the algorithm. The generating function $\mu(z)$ for the sequence $\{ \E{S_n}; n \geq 0 \}$ is
\begin{equation}
\mu(z) = \frac{qz\Pk{0}(z)}{1-z} + \frac{(p-q)z}{(1-z)^2}.
\label{eq:mu_n_GF}
\end{equation}
The asymptotic form of $\mu_n \equiv \nthCoeff \mu(z) = \E{S_n}$ is given by
\begin{equation*}
\mu_n \approx \begin{dcases}
\frac{p}{q-p}, & p < \frac{1}{2} \\
\sqrt{\frac{2n}{\pi}} - \frac{1}{2}, & p = \frac{1}{2} \\
(p-q)n + \frac{q}{p-q}, & p > \frac{1}{2}.
\end{dcases}
\end{equation*}

The error bound for the first and third cases is $O((2\sqrt{pq}+\varepsilon)^n)$. When $p = 1/2$, the error bound is $O(n^{-1/2})$.

\begin{proof}
Since $\mu(z) \equiv \left. \partial_u P(z,u) \right|_{u=1}$, the expression \eqref{eq:mu_n_GF} can be obtained from \eqref{eq:Pnk_GF} in a straightforward manner. 

For the asymptotic analysis, we cover only the case $p = 1/2$ since the result will be used again later. Here, the function $\mu(z)$ simplifies to
\begin{equation*}
\mu(z) = \frac{1}{2} \left[  \sqrt{\frac{1+z}{(1-z)^3}} - \frac{1}{1-z} \right],
\end{equation*}
implying
\begin{equation}
\mu_n =  \frac{1}{2} \nthCoeff \left\lbrace \frac{(1+z)^{1/2}}{(1-z)^{3/2}} \right\rbrace - \frac{1}{2}.
\label{eq:mu_n_phalf_1}
\end{equation}
In the neighborhood of the singularities of $\mu(z)$, we have
\begin{equation*}
\frac{(1+z)^{1/2}}{(1-z)^{3/2}} =
\begin{dcases}
\frac{\sqrt{2}}{(1-z)^{3/2}} + O\left( (1-z)^{-1/2} \right), & z \rightarrow 1 \\
 O\left( \sqrt{1+z} \right), & z \rightarrow -1,
\end{dcases}
\end{equation*}
and therefore by Big-Oh transfer
\begin{equation}
\nthCoeff \left\lbrace \frac{(1+z)^{1/2}}{(1-z)^{3/2}} \right\rbrace = 2\sqrt{\frac{2n}{\pi}} + O\left( n^{-1/2} \right).
\label{eq:bigOhtransm1} 
\end{equation}
Substituting this result into \eqref{eq:mu_n_phalf_1} gives us
\begin{equation*}
\mu_n = \sqrt{\frac{2n}{\pi}} - \frac{1}{2} + O\left( n^{-1/2} \right).
\end{equation*}
\end{proof}

\label{prop:Sn_1facmoment} 
\end{proposition}

\begin{proposition}
Let $S_n$ denote the stratum number of the tree generated by the $n$th iteration of the algorithm. The generating function $\mu^{(2)}(z)$ for the second factorial moment of $S_n$, namely $\E{S_n(S_n-1)}$, is
\begin{equation}
\mu^{(2)}(z) = \frac{2q(2pz-1)z\Pk{0}(z)}{(1-z)^2} + \frac{2(4p-3)pz^2}{(1-z)^3} + \frac{2qz}{(1-z)^3}.
\label{eq:mu2_n_GF}
\end{equation}
The asymptotic form of $\mu^{(2)}_n \equiv \nthCoeff \mu^{(2)}(z) = \E{S_n(S_n-1)}$ is given by

\begin{equation*}
\mu^{(2)}_n \approx \begin{dcases}
2 \left( \frac{p}{q-p} \right)^2, & p < \frac{1}{2} \\
n + 1 - 2\sqrt{\frac{2n}{\pi}}, & p = \frac{1}{2} \\
(p-q)^2n^2 - (4p^2-3)n + \frac{2q(1-3p)}{(p-q)^2}, & p > \frac{1}{2}.
\end{dcases}
\end{equation*}

The error bound for the first and third cases is $O((2\sqrt{pq}+\varepsilon)^n$ and $O\left( n^{-1/2} \right)$ when $p=1/2$.

\label{prop:Sn_2facmoment}
\end{proposition}

\begin{proof}
Equation \eqref{eq:mu2_n_GF} follows directly from the relation $\mu^{(2)}(z) = \left. \partial^2_u P(z,u) \right|_{u=1}$. 

\end{proof}

\begin{proposition}
Let $S_n$ denote the stratum number of the tree generated by the $n$th iteration of the algorithm. The asymptotic form of the variance of $S_n$ is given by
\begin{equation*}
\Var{S_n} \approx \begin{dcases}
\frac{pq}{(q-p)^2}, & p < \frac{1}{2} \\ 
\left(1-\frac{2}{\pi} \right)n - \frac{1}{2}\sqrt{\frac{2n}{\pi}} + \frac{1}{4}\left(3 - \frac{1}{\pi}\right) , & p = \frac{1}{2} \\ 
\frac{pq[4(1-4pq)n-3]}{(p-q)^2}, & p > \frac{1}{2}.
\end{dcases}
\end{equation*}

The error bound is $O((2\sqrt{pq}+\varepsilon)^n)$ for the first case, $O(n^{-1/2})$ for the second and $O((2\sqrt{pq}+\varepsilon)^n \, n)$ for the third.

\begin{proof}
Noting $\Var{S_n} = \mu_n^{(2)} + \mu_n - (\mu_n)^2$, the result for cases $p < 1/2$ and $p > 1/2$ is an immediate consequence of Propositions~\ref{prop:Sn_1facmoment} and \ref{prop:Sn_2facmoment}. When $p = 1/2$, applying those propositions leads to an $O(\sqrt{n})$ error bound. In order to get a vanishing error bound, we need to expand \eqref{eq:bigOhtransm1} by an additional term, namely
\begin{equation*}
\frac{(1+z)^{1/2}}{(1-z)^{3/2}} =
\begin{dcases}
\frac{\sqrt{2}}{(1-z)^{3/2}} + \frac{1}{2\sqrt{2}} \cdot \frac{1}{\sqrt{1-z}}+O\left( \sqrt{1-z} \right), & z \rightarrow 1 \\
 O\left( \sqrt{1+z} \right), & z \rightarrow -1
\end{dcases}
\end{equation*}
which yields the refined asymptotic estimate of $\E{S_n}$ when $p = 1/2$,
\begin{equation*}
\mu_n = \sqrt{\frac{2n}{\pi}} - \frac{1}{2} + \frac{1}{4\sqrt{2\pi n}} + O(n^{-3/2}).
\end{equation*} 

The result now follows in the same manner as the other cases.

\end{proof}

\label{prop:Sn_var}
\end{proposition}

\begin{lemma}
The generating function $H(z)$ for the expected value of $H_{S_n}$, where $H_0 = 0$ and $H_k$ denotes the $k$th harmonic number $(k \geq 1)$, is
\begin{equation}
H(z) = \frac{1}{1-z} \, \log{\left(\frac{1+\sqrt{1-4pqz^2}}{1-2pz+\sqrt{1-4pqz^2}}\right) }.
\label{eq:harnum_gf}
\end{equation}

The asymptotic form of $\nthCoeff H(z) = \E{H_{S_n}}$ is given by
\begin{equation*}
\E{H_{S_n}} \approx \begin{dcases}
\log \left( \frac{q}{q-p} \right), & p < \frac{1}{2} \\
\log \sqrt{n}, & p = \frac{1}{2} \\
\log(p-q) + \log n & p > \frac{1}{2}. 
\end{dcases}
\end{equation*}

\begin{proof}
We first note the desired expectation $\E{H_{S_n}}$ can be written as
\begin{equation*}
\sum_{k \geq 1} H_k \Prob{S_n = k} = \sum_{k \geq 1}  \sum_{j = 1}^k \frac{1}{j} \, \Prob{S_n = k} = \sum_{j \geq 1} \frac{1}{j} \sum_{k \geq j} \Prob{S_n = k} = \sum_{j \geq 1} \frac{1}{j} \Prob{S_n \geq j}.
\end{equation*}

Thus if we can find the bivariate generating function
\begin{equation*}
F(z,u) = \sum_{n \geq 0} \sum_{k \geq 0} \Prob{S_n \geq k} u^k z^n,
\end{equation*}
the desired generating function is
\begin{equation*}
H(z) = \int_0^1 \frac{1}{s} \left[ F(z,s) - \frac{1}{1-z} \right]  \dx{s}.  
\end{equation*}

To that end, we derive
\begin{align}
F(z,u) &= \sum_{n \geq 0} \sum_{k \geq 0} \Prob{S_n \geq k} u^k z^n \nonumber \\
&= \sum_{n \geq 0} \sum_{k \geq 0} u^kz^n - \sum_{n \geq 0} \sum_{k \geq 0} \Prob{S_n \leq k-1} u^k z^n.
\label{eq:Fzu}
\end{align}

Focusing on the second term of \eqref{eq:Fzu}, we find
\begin{align}
\sum_{n \geq 0} \sum_{k \geq 0} \Prob{S_n \leq k-1} u^k z^n &= \sum_{n \geq 0} \sum_{k \geq 1} \Prob{S_n \leq k-1} u^k z^n \nonumber \\
&= u \sum_{n \geq 0} \sum_{k \geq 0} \Prob{S_n \leq k} u^k z^n \nonumber \\
&= u \sum_{n \geq 0} \sum_{k \geq 0} [u^k] \frac{P_n(u)}{1-u} u^k z^n \nonumber & \text{where\ } P_n(u) \equiv \nthCoeff P(z,u) \\
&= \frac{uP(z,u)}{1-u},
\label{eq:S_nCumGF}
\end{align}
since $P_n(u)$ is the probability generating function of $S_n$ and therefore $P_n(u)(1-u)^{-1}$ is the generating function of the sequence $\{ \Prob{S_n \leq k}; k \geq 0 \}$.

Substituting \eqref{eq:S_nCumGF} into \eqref{eq:Fzu} yields
\begin{equation*}
F(z,u) = \frac{1}{(1-u)(1-z)} - \frac{uP(z,u)}{1-u},
\end{equation*}
and so
\begin{equation*}
H(z) = \int_0^1 \left[ \frac{1}{(1-s)(1-z)} - \frac{P(z,s)}{1-s} \right] \dx{s}.
\end{equation*}

We conclude the derivation of \eqref{eq:harnum_gf} by observing
\begin{equation*}
P(z,s) = \frac{A-(A+1)s}{pz(s-B)(s-C)},
\end{equation*}
where
\begin{align*}
A = qz\Pk{0}(z), && B = \frac{1+\sqrt{1-4pqz^2}}{2pz},  && C = \frac{1-\sqrt{1-4pqz^2}}{2pz}.
\end{align*}
The integral (with respect to $s$) follows immediately after a partial fractions expansion of $P(z,s)(1-s)^{-1}$ and some simplification of the result.

\end{proof}
\label{lem:harnum_gf}
\end{lemma}

\begin{lemma}
Let $Z_n$ denote the size of the tree generated by the $n$th iteration of the algorithm, ie, $Z_n = S_n + 1$. The generating function $h(z)$ for the mean of the reciprocal of $Z_n$, that is, $\E{Z_n^{-1}}$ is
\begin{equation}
h(z) = \frac{1+\sqrt{1-4pqz^2}}{pz[1-2qz+\sqrt{1-4pqz^2}]} \, \log{\left(\frac{1+\sqrt{1-4pqz^2}}{1-2pz+\sqrt{1-4pqz^2}}\right) }.
\label{eq:recip_exp_size_gf}
\end{equation}

The asymptotic form of $\nthCoeff h(z) = \E{Z_n^{-1}}$ is given by
\begin{equation*}
\E{Z_n^{-1}} \approx \begin{dcases}
\frac{q-p}{p} \log \left( \frac{q}{q-p} \right) , & p < \frac{1}{2} \\
\frac{\log n}{\sqrt{2\pi n}}, & p = \frac{1}{2} \\
\frac{1}{(p-q)n}, & p > \frac{1}{2}. 
\end{dcases}
\end{equation*}

\begin{proof}
It is straightforward to show $h(z) = \int_0^1 P(z,s) \dx{s}$ and using the partial fractions expansion outlined in Lemma~\ref{lem:harnum_gf}, the integral leads to \eqref{eq:recip_exp_size_gf}. 
 
\end{proof}

\label{lem:recip_treesize_gf}
\end{lemma}

\begin{proposition}
Let $H_{Z_n}$ denote the $Z_n$th harmonic number, where $Z_n=S_n+1$ is the tree size after the $n$th iteration of the algorithm. The generating function of $\E{H_{Z_n}}$ is $H(z)+h(z)$ as given in lemmas~\ref{lem:harnum_gf} and \ref{lem:recip_treesize_gf}. The asymptotic estimates of $\E{H_{Z_n}}$ are
\begin{equation*}
\E{H_{Z_n}} \approx \begin{dcases}
\frac{q}{p} \log \left( \frac{q}{q-p} \right) , & p < \frac{1}{2} \\
\log \sqrt{n} + \frac{\log n}{\sqrt{2\pi n}}, & p = \frac{1}{2} \\
\log(p-q) + \log n + \frac{1}{(p-q)n}, & p > \frac{1}{2}. 
\end{dcases}
\end{equation*}

\begin{proof}
The results follow immediately from the identity $H_{Z_n} = H_{S_n} + \frac{1}{Z_n}$.
\end{proof}

\label{prop:harmonic_num_gf}
\end{proposition}

\section{Application of Results to Tree Functionals}

\subsection{Tree Size}

Given we use uniform attachment for the addition rule and any member of the class defined in Corollary~\ref{cor:eqprobclass} for deletion, we immediately have from Propositions~\ref{prop:Sn_1facmoment} and \ref{prop:Sn_var} the asymptotics of the expected tree size  and corresponding variance of tree $T_n$, namely $\E{Z_n} = \E{S_n}+1$ and $\Var{Z_n} = \Var{S_n}$. We note there are three behavioral regimes determined by whether insertion probability $p$ is less than, equal to, or exceeds $1/2$.

\subsection{Leaf Count}

By conditioning on stratum number, we can obtain similar results for other tree functionals. To see this, suppose the distribution of $T_n$, conditioned on $S_n=k$, is equiprobable. That is,
\begin{equation*}
\mathbb{P}_q \left\lbrace T_n=t \given S_n=k \right\rbrace  = 
\begin{dcases}
\frac{1}{k!},& t \text{\ is a stratum\ } k \text{\ tree} \\
0,& \text{otherwise},
\end{dcases}
\end{equation*}

where $\mathbb{P}_x$ is the probability of an event given a deletion probability of $x$. Observe $\mathbb{P}_0 \left\lbrace T_k = t \right\rbrace = \frac{1}{k!}$ whenever $t$ is a stratum $k$ tree and is 0 otherwise. The conditional expectation of $f(T_n)$ is thus given by
\begin{align*}
\mathbb{E}_q \left[ f(T_n) \given S_n=k \right] &= \sum_{t \in \mathcal{T}} f(t) \mathbb{P}_q \left\lbrace T_n=t \given S_n=k \right\rbrace \\
&= \sum_{t \in \mathcal{T}_k } f(t) \mathbb{P}_q \left\lbrace T_n=t \given S_n=k \right\rbrace = \sum_{t \in \mathcal{T}_k } f(t) \mathbb{P}_0 \left\lbrace T_k = t \right\rbrace
= \mathbb{E}_0 \left[ f(T_k) \right], 
\end{align*}
where $\mathcal{T}$ and $\mathcal{T}_k \subseteq \mathcal{T}$ denote the set of all possible trees and the subset of stratum $k$ trees, respectively, and $\mathbb{E}_x$ is the expectation of an event given a deletion probability of $x$.

Since the expectation $\mathbb{E}_0 \left[ f(T_k) \right]$ depends only on $k$, we can ``ignore'' the effect of deletion probability $q$ on the probabilistic behavior of the tree functional and, by iterating expectation, exploit the useful result
\begin{equation}
\mathbb{E}_q \left[ f(T_n) \right] = \mathbb{E}_q \left[ \mathbb{E}_0 \left[ f(T_{S_n}) \right] \right].
\label{eq:ignore_q}
\end{equation}
Interestingly, this results holds regardless of the specific deletion rule from the class chosen. Whether it is \textit{LIFO} or something more intricate, the moments are the same.

If we let $L_n$ denote the number of leaves in the tree at time $n$, we have the well-known result \cite[pp 326-327]{HOFR18}
\begin{equation*}
\mathbb{E}_0 \left[ L_n \right] = \frac{n+1}{2} \Indicator{n>0} + \Indicator{n=0},
\end{equation*}
where $\Indicator{\cdot}$ is Iverson bracket notation for an indicator function. By using \eqref{eq:ignore_q} we can deduce
\begin{equation*}
\mathbb{E}_q [L_n] = \frac{1+\mathbb{E}_q [S_n] + \mathbb{P}_q \Event{S_n=0} }{2}.
\end{equation*}

Applying Propositions~\ref{prop:Pn0_GF} and \ref{prop:Sn_1facmoment} yields the generating function and asymptotic estimate.

\subsection{Root Degree}

Let $D_n$ denote the degree of the root node at time $n$. Since
\begin{equation*}
\mathbb{E}_0 \left[ D_n \right] = H_{n+1},
\end{equation*}
where $H_n$ is the $n$th harmonic number \cite[pp 323-324]{HOFR18}, we find
\begin{equation*}
\mathbb{E}_q [D_n] = \mathbb{E}_q [H_{S_n+1}] = \mathbb{E}_q [H_{Z_n}],
\end{equation*}
and obtain the corresponding generating function and asymptotics from Proposition~\ref{prop:harmonic_num_gf}.

\section{Summary}

By allowing for the possibility of node removal during the course of their evolution, we can extend the utility of random recursive trees models. The analysis of such trees; however, is complicated by the fact that the tree size at time $n$ is no longer deterministic. Nevertheless, for the class of deletion rules identified by Corollary~\ref{cor:eqprobclass}, we showed the current tree $T_n$ conditioned on its size is uniformly distributed over its range. This reduces the problem of studying $T_n$ to that of studying its stratum number. By using generating function theory, we obtain several results for the expected tree size, leaf count and root degree of tree $T_n$.

\bibliographystyle{amsplain}
\bibliography{rtreeswdel}

\providecommand{\bysame}{\leavevmode\hbox to3em{\hrulefill}\thinspace}
\providecommand{\MR}{\relax\ifhmode\unskip\space\fi MR }
\providecommand{\MRhref}[2]{%
  \href{http://www.ams.org/mathscinet-getitem?mr=#1}{#2}
}
\providecommand{\href}[2]{#2}
\begin{thebibliography}{10}

\bibitem{BENN07}
E.~Ben-Naim and P.L. Krapivsky, \emph{Addition-deletion networks}, The Journal
  of Physics A: Mathematical and Theoretical \textbf{40} (2007), 8607–8619.

\bibitem{DEON07}
Narsingh Deo and Aurel Cami, \emph{Preferential deletion in dynamic models of
  web-like networks}, Information Processing Letters \textbf{102} (2007),
  156--162.

\bibitem{DOND82}
Marian Dondajewski and Jerzy Szymanski, \emph{On the distribution of
  vertex-degrees in a strata of a random recursive tree}, Bulletin de
  l'Acad\'{e}mie Polonaise des Sciences. S\'{e}rie des Sciences
  Math\'{e}matiques \textbf{30} (1982), 205--209.

\bibitem{FLAJ09}
Philippe Flajolet and Robert Sedgewick, \emph{Analytic combinatorics}, 3rd ed.,
  Cambridge University Press, 2009.

\bibitem{GHOS13}
Gourab Ghoshal, Liping Chi, and Albert-L\'{a}szl\'{o} Barab\'{a}si,
  \emph{Uncovering the role of elementary processes in network evolution},
  Scientific Reports \textbf{3} (2013), no.~2920.

\bibitem{HOFR18}
Micha Hofri and Hosam Mahmoud, \emph{Algorithms of nonuniformity: Tools and
  paradigms}, CRC Press, 2018.

\bibitem{JANS05}
Svante Janson, \emph{Asymptotic degree distribution in random recursive trees},
  Random Structures and Algorithms \textbf{26} (2005), 69--83.

\bibitem{JOHA18}
Tony Johansson, \emph{Deletion of oldest edges in a preferential attachment
  graph (manuscript)}, 2018.

\bibitem{JONA78}
Arne~T. Jonassen and Donald~E. Knuth, \emph{A trivial algorithm whose analysis
  isn't}, Journal of Computer and System Sciences \textbf{16} (1978), no.~3,
  301--322.

\bibitem{NAJO82}
D~Najock and C.C. Heyde, \emph{On the number of terminal vertices in certain
  random trees with an application to stemma construction in philology},
  Journal of Applied Probability \textbf{19} (1982), 675–680.

\bibitem{PANN10}
Wolfgang Panny, \emph{Deletions in random binary search trees: a story of
  errors}, Journal of Statistical Planning and Inference \textbf{140} (2010),
  no.~8, 2335--2345.

\bibitem{PITT94}
Boris Pittel, \emph{Note on the height of random recursive trees and m-ary
  search trees}, Random Structures and Algorithms - RSA \textbf{5} (1994),
  337--347.

\bibitem{SMYT95}
Robert~T. Smythe and Hosam~M. Mahmoud, \emph{A survey of recursive trees},
  Theory of Probability and Mathematical Statistics \textbf{51} (1995), 1--27.

\bibitem{ZHAN16}
Xiaojun Zhang, Zheng He, and Lez Rayman-Bacchus, \emph{Random birth-and-death
  networks}, Journal of Statistical Physics \textbf{162} (2016), no.~4,
  842--854.

\end{thebibliography}

\end{document}